\documentclass{article}
\usepackage{amsmath}
\usepackage{amsfonts}
\usepackage{amssymb}
\usepackage{amsthm}
\usepackage[shortlabels]{enumitem}

\title{Comment to ``Almost disjoint sets, the dense set problem and the partition calculus''}
\author{Júnio Luan Pereira \\ Independent Researcher}
\date{}

\newcommand{\stick}{%
\ensuremath{%
{%
\mathchoice%
{{\text{\raisebox{0.7ex}{$\bullet$}\kern-0.39em$|$}}}%
{{\text{\raisebox{0.7ex}{$\bullet$}\kern-0.39em$|$}}}%
{{\text{\raisebox{0.65ex}{$\bullet$}\kern-0.40em$|$}}}%
{{\text{\raisebox{0.6ex}{$\bullet$}\kern-0.42em$|$}}}%
}%
}%
}

\newtheorem{definition}{Definition}
\newtheorem{proposition}[definition]{Proposition}

\begin{document}
\maketitle

\begin{abstract}
  This text highlights issues present in the proof of Lemma 6.10 of the Baumgartner ($\ast$1943, $\dag$2011)
  article ``Almost disjoint sets, the dense set problem and the partition calculus'' of 1976, and intends to
  present a correction at the same time it proves a stronger result mentioned in the article to have similar proof.
\end{abstract}

\section{Introduction}

Lemma 6.10 from Baumgartner's article \cite[p.\@ 428]{baumgartner_almost-disjoint_1976} belongs to the proof of
Theorem 6.7, a consistency result valid for any regular cardinal $\kappa$. The case $\kappa = \aleph_0$ of Theorem
6.7 is the literature reference for the consistency with $\neg \text{CH}$ of the statement now denoted as $\stick =
\aleph_1$ and described as \textbf{stick principle}, although Baumgartner's proof itself focuses on uncountable
cardinals and just mentions that case $\kappa = \aleph_0$ can be treated as if it is an inaccessible cardinal. At the
end of \S6, the article also proves a stronger statement by means of \cite[(25), p.\@
433]{baumgartner_almost-disjoint_1976}, which proof is mentioned to be done by complicating the proof of Lemma 6.10.

The proof presented in \cite{baumgartner_almost-disjoint_1976} for Lemma 6.10, however, has issues that are not
merely mistypes, and there is also some difficulties that turn nontrivial to apply it for $\kappa = \aleph_0$ in the
way the proof suggests. The purpose of this comment is to describe and address such issues while providing a complete
proof for \cite[(25), p.\@ 433]{baumgartner_almost-disjoint_1976}.

The case $\beta = \kappa$ of \cite[(25), p.\@ 433]{baumgartner_almost-disjoint_1976} has a proof that is identical to
Lemma 6.10, then we will prove it first, mention the changes required to apply it for Lemma 6.10, and use it to list
the issues in the original proof. The Lemma 6.10 will be required in order to complete the proof of \cite[(25), p.\@
433]{baumgartner_almost-disjoint_1976}, that will follow the same lines as the case $\beta = \kappa$.

The Section \ref{secprelim} will introduce all the things from \cite{baumgartner_almost-disjoint_1976} required in
order to understand the proof of \cite[(25), p.\@ 433]{baumgartner_almost-disjoint_1976}, that will be presented in
Section \ref{secmain} together with the list of issues from the original proof.

We will assume all the conventions of notation and definition from \S1 of Baumgartner's article, described in
\cite[p.\@ 401-406]{baumgartner_almost-disjoint_1976}. Although we will make considerable changes in the structure of
the arguments from the proof in order to turn its main ideas more visible, we will try, at the extent of the
possible, to preserve all the terminology used in \cite{baumgartner_almost-disjoint_1976} here. We will also give
emphasis to explicitly referencing all statements and definitions from the original article that will be replicated
here.

\section{Preliminaries}\label{secprelim}

The Lemma 6.10 is a forcing technique consistence result over countable transitive models. The partial order used in
the forcing technique is the following one.
\begin{definition}[{\cite[p.\@ 428]{baumgartner_almost-disjoint_1976}}]
  Let $\kappa$ and $\lambda$ be cardinals, with $\kappa$ being regular. $R(\kappa, \lambda)$ is the partially ordered
  set of all subsets $B$ of $\lambda \times 2 \times \kappa$ satisfying:
  \begin{enumerate}[(1), series=Baumgartner]
    \item $|B| \leq \kappa$ \cite[(13), p.\@ 428]{baumgartner_almost-disjoint_1976},

    \item if $\alpha < \kappa$ and $\beta < \kappa$ then either $(\alpha, 0, \beta) \not\in B$ or $(\alpha, 1,
        \beta) \not\in B$ \cite[(14), p.\@ 428]{baumgartner_almost-disjoint_1976},

    \item \label{baumrklreq3} for all $\alpha < \kappa$, $\{\beta < \kappa: (\alpha, 0, \beta) \not\in B, (\alpha,
        1, \beta) \not\in B\}$ is closed and unbounded in $\kappa$ \cite[(15), p.\@
        428]{baumgartner_almost-disjoint_1976};
  \end{enumerate}
  such that $B_1 \leq B_2$ iff $B_2 \subseteq B_1$ for all $B_1, B_2 \in R(\kappa, \lambda)$.
\end{definition}
If $\kappa = \aleph_0$, then ``closed and unbounded'' means merely ``infinite''.

Each $R(\kappa, \lambda)$ has the $(2^\kappa)^+$-chain condition \cite[Lemma 6.9, p.\@
429]{baumgartner_almost-disjoint_1976} and, for every uncountable $\kappa$, $R(\kappa, \lambda)$ is $\kappa$-closed
\cite[Lemma 6.8, p.\@ 429]{baumgartner_almost-disjoint_1976} due to the validity of:
\begin{enumerate}[resume*=Baumgartner]
  \item \label{baumkclosed} The intersection of fewer than $\kappa$ closed unbounded sets of $\kappa$ is itself
      closed unbounded, hence non-empty \cite[p.\@ 430]{baumgartner_almost-disjoint_1976}.
\end{enumerate}
For $\kappa = \aleph_0$, the statement above is false, but it can be trivially proved that $R(\aleph_0, \lambda)$ is
$\aleph_0$-closed too.

If $\mathfrak{M}$ is a countable transitive model of ZFC and $G$ is ${R(\kappa, \lambda)}^{\mathfrak{M}}$-generic
over $\mathfrak{M}$, then $G$ provides directly a new function $H: \lambda \rightarrow {{^2}\kappa}$ for the
countable transitive model $\mathfrak{M}[G]$. Additionally, $\mathfrak{M}[G]$ have the new sets:
\begin{enumerate}[resume*=Baumgartner]
  \item $G_\alpha = \{\beta < \kappa: (\exists B \in G) (\alpha, 0, \beta) \in B\} \not\in \mathfrak{M}$ for every
      $\alpha < \lambda$, with $G_\alpha \neq G_{\alpha'}$ whenever $\alpha \neq \alpha'$ \cite[p.\@
      428]{baumgartner_almost-disjoint_1976}.
\end{enumerate}
Thus, if $\mathfrak{M}[G]$ preserve cardinalities, then $2^\kappa \geq \lambda$ in $\mathfrak{M}[G]$.

For every $B \in R(\kappa, \lambda)$, we define the domain of $B$ to be
$$\mathrm{domain}\ B = \{\alpha < \lambda: (\exists i < 2)(\exists \beta < \kappa)(\alpha, i, \beta) \in B\}$$
\cite[p.\@ 429]{baumgartner_almost-disjoint_1976}.

\section{Corrected version of Baumgartner's proof}\label{secmain}

The proof of Theorem 6.7 ends with the proof of \cite[Lemma 6.11, p.\@ 432]{baumgartner_almost-disjoint_1976}. The
strengthening of Lemma 6.11 demonstrated in \cite[(c), p.\@ 432]{baumgartner_almost-disjoint_1976} uses the following
proposition.
\begin{proposition}
  Let $\mathfrak{M}$ be a countable transitive model of ZFC + GCH, and let $\kappa$ and $\lambda$ be cardinals in
  $\mathfrak{M}$ such that $\kappa$ is regular and $\kappa \leq \lambda$. If $\kappa$ is not inaccessible in
  $\mathfrak{M}$, then assume also that $\diamond_\kappa$ holds in $\mathfrak{M}$. Let $Z$ be a set of ordinals
  well-ordered by ${\prec}$ in $\mathfrak{M}$ and $\beta < \kappa^+$. If $G$ is ${R(\kappa,
  \lambda)}^{\mathfrak{M}}$-generic over $\mathfrak{M}$, then
  \begin{enumerate}[resume*=Baumgartner]
    \item If $Y \in \mathfrak{M}[G]$ is a subset of $Z$ which has order-type $\kappa^+$ with respect to $\prec$,
        then there is $X\subseteq Y$ such that $X \in \mathfrak{M}$ and $X$ has order-type $\beta$ with respect to
        $\prec$. \cite[(25), p.\@ 433]{baumgartner_almost-disjoint_1976}.
  \end{enumerate}
\end{proposition}
In its proof, mentioned by \cite{baumgartner_almost-disjoint_1976} to be a complicated version of the proof of Lemma
6.10, we will not use the fact that the elements of $Z$ are ordinals but only that both $Z$ and $\prec$ belongs to
the countable transitive model $\mathfrak{M}$. Since for every $z \in Z$ the partial order
$$s_\prec(z) = \{x \in Z : x \prec z\} \  \text{ordered by} \  {\prec}$$
will also belong to $\mathfrak{M}$, we can assume without loss of generality that, in $\mathfrak{M}[G]$, $Y$ is
unbounded in $Z$.

Given such an $Y \in \mathfrak{M}[G]$, let $\dot{Y} \in \mathfrak{M}$ be a name for it. For every $B \in R(\kappa,
\lambda)$, let
$$\mathrm{obj}\ B = \{x \in Z: B \Vdash \check{x} \in \dot{Y}\}.$$
In order to prove the Proposition above, it will be enough to ensure in $\mathfrak{M}$ that, for every $\beta <
\kappa^+$ and $B \in R(\kappa, \lambda)$ satisfying
$$B \Vdash \dot{Y} \subseteq \check{Z}, \  \dot{Y} \  \text{is unbounded in} \  (\check{Z}, \check{\prec})\
\text{and}\ \text{order-type}\ (\dot{Y}, \check{\prec}) = {(\kappa^+)}\,{\check{}}$$ ($\kappa^+$ is meant to be only
an ordinal here), there exists $B' \leq B$ such that $(\mathrm{obj}\ B', \prec)$ has order-type $\geq \beta$. From
now on, let $B$ be any member of $R(\kappa, \lambda)$ satisfying the statement above, which we can assume
additionally that $|\mathrm{domain}\ B| = \kappa$.

The case $\beta = \kappa$ has proof identical to Lemma 6.10.
\begin{proof}[Proof for $\beta = \kappa$]
  Choose $A \subseteq \kappa$ belonging to $\mathfrak{M}$ such that $|A| = |\kappa \setminus A| = \kappa$ with a
  $\kappa$-partition in $\mathfrak{M}$ of $\kappa \setminus A$ $(A_\alpha: \alpha < \kappa)$, all of whose $A_\alpha$
  are $\kappa$-sized. In $\mathfrak{M}$, we will construct by transfinite recursion the sequence $(B_\alpha: \alpha <
  \kappa)$ of elements of $R(\kappa, \lambda)$, the sequence $(f_\alpha: \alpha < \kappa)$ of functions, the sequence
  $(E_\alpha: \alpha < \kappa)$ of subsets of $\lambda$ and the sequence $(F_\alpha: \alpha < \kappa)$ of subsets of
  $\kappa$ satisfying:
  \begin{enumerate}[resume*=Baumgartner]
    \item $B_0 = B$ \cite[(17), p.\@ 429]{baumgartner_almost-disjoint_1976};
    \item \label{baumfalpha} For all $\alpha < \kappa$, $f_\alpha$ maps $\alpha \cup A \cup \bigcup \{A_\xi: \xi <
        \alpha\}$ 1-1 onto $\mathrm{domain}\ B_\alpha$ \cite[(18), p.\@ 429]{baumgartner_almost-disjoint_1976}.
    \item For all $\alpha < \kappa$, $E_\alpha = \{f_\alpha(\xi) : \xi < \alpha\}$ \cite[p.\@
        430]{baumgartner_almost-disjoint_1976};
    \item \label{baumincreq} If $\alpha < \gamma < \kappa$, then $B_\alpha \lneq B_\gamma$, $f_\alpha \subsetneq
        f_\gamma$ and $F_\alpha \subsetneq F_\gamma$
    \item If $\alpha$ is a limit ordinal, then
        $$B_\alpha = \bigcup_{\xi < \alpha} B_\xi, \quad f_\alpha = \bigcup_{\xi < \alpha} f_\xi, \quad F_\alpha =
        \bigcup_{\xi < \alpha} F_\xi.$$
    \item \label{baumrecsucreq} For every $\alpha < \kappa$,
        $${B_\alpha \cap (E_\alpha \times 2 \times F_{\alpha + 1})} = {B_{\alpha + 1} \cap (E_\alpha \times 2 \times
        F_{\alpha + 1})},$$ \cite[p.\@ 429-430, inside both Cases 1 and 2]{baumgartner_almost-disjoint_1976}.
  \end{enumerate}
  The requirements above already give us consistently how the recursion step must be if $\alpha$ is limit ordinal.
  The $E_\alpha$ is already fully defined in each step of the recursion. In order to conclude the recursive
  construction requirements, we must include additional requirements it must satisfy depending on the regular
  $\kappa$, that will be divided in three cases.

  \textit{Case 1}: $\kappa$ is uncountable and inaccessible. Here, we must construct additionally the sequence
  $(a_\alpha: \alpha < \kappa)$ of elements of $\kappa$ satisfying:
  \begin{enumerate}[resume*=Baumgartner]
    \item \label{baumfandaalpha} $F_\alpha = \{a_\xi: \xi < \alpha\}$;
    \item \label{baumaalpha} If $\gamma < \alpha < \kappa$, then $(f_\alpha(\gamma), 0, a_\alpha),
        (f_\alpha(\gamma), 1, a_\alpha) \not\in B_\xi$ for all $\xi \leq \alpha$ \cite[p.\@ 430, implicit in both
        Cases 1 and 2]{baumgartner_almost-disjoint_1976};
    \item \label{baumlimaalpha} If $\gamma < \alpha < \kappa$, then $a_\gamma < a_\alpha$ and, if $\alpha$ is an
        limit ordinal, then
        $$a_\alpha = \sup_{\xi < \alpha} a_\xi.$$
  \end{enumerate}
  The requirements above already define consistently $a_\alpha$ if $\alpha$ is limit ordinal and turn $F_\alpha$
  fully defined in each step of the recursion. It is true here that $|\wp(E_\alpha \times 2 \times F_{\alpha + 1})| <
  \kappa$ for all $\alpha < \kappa$, and that \ref{baumrecsucreq} combined with \ref{baumaalpha} implies
  \begin{enumerate}[resume*=Baumgartner]
    \item \label{baumextaalpha} If $\gamma < \alpha < \kappa$, then $(f_\alpha(\gamma), 0, a_\alpha),
        (f_\alpha(\gamma), 1, a_\alpha) \not\in B_\xi$ for all $\xi < \kappa$ \cite[(19), p.\@
        429]{baumgartner_almost-disjoint_1976}.
  \end{enumerate}
  Let us proceed with the recursive construction in this case.

  Let $B_0 = B$. Choose arbitrarily an 1-1 function $f_0 : A \rightarrow \mathrm{domain}\ B_0$ and an ordinal $a_0 <
  \kappa$. This is sufficient for the step $\alpha = 0$.

  In the $(\alpha + 1)$th step, fix $(D_\xi: \xi < \tau_\alpha)$ an enumeration of all sets $D$ such that
  $${B_\alpha \cap {(E_\alpha \times 2 \times F_{\alpha + 1})}} \subseteq {D} \subseteq {(E_\alpha \times 2 \times
  F_{\alpha + 1})},$$ hence $\tau_\alpha < \kappa$. To define $B_{\alpha + 1}$, the recursion will construct
  additionally $x^\alpha_\xi \in Z$ for (not necessarily all) pair $\alpha < \kappa$ and $\xi < \tau_\alpha$ in a way
  that the following requirement is satisfied:
  \begin{enumerate}[resume*=Baumgartner]
    \item \label{bauminacunc} If there exists $B' \leq B_{\alpha + 1}$ and $x \in Z$ different from all defined
        $x^\gamma_\eta$ with $\gamma \leq \alpha$ and $\eta < \tau_\alpha$, both satisfying
        $${B' \cap {(E_\alpha \times 2 \times F_{\alpha + 1})} = D_\xi} \  {\text{and}} \  B' \Vdash \check{x} \in
        \dot{Y},$$ then $x^\alpha_\xi$ is defined and
        $$B_{\alpha + 1} \cup D _\xi \Vdash {(x^\alpha_\xi)\,\check{}} \in \dot{Y}.$$
  \end{enumerate}
  In order to perform it we must construct recursively the $x^\alpha_\xi$ together with the sequence $(C_\xi: \xi
  \leq \tau_\alpha)$ as follows. Let $C_0 = B_\alpha \setminus {(E_\alpha \times 2 \times F_{\alpha + 1})}$. Given
  $C_\xi$ with $\xi < \tau_\alpha$, if there are $\bar{B} \in R(\kappa, \lambda)$ and $x \in Z$ distinct of all
  $x^\gamma_\eta$ defined so far (i.e.\@ with either $\gamma < \alpha$ or $\gamma = \alpha$ and $\eta < \xi$)
  satisfying
  $$\bar{B} \leq C_\xi, \ \bar{B} \cap {(E_\alpha \times 2 \times F_{\alpha + 1})} = D_\xi \ \text{and} \ \bar{B}
  \Vdash \check{x} \in \dot{Y},$$ then let $x^\alpha_\xi = x$ and $C_{\xi + 1} = \bar{B} \setminus {(E_\alpha \times
  2 \times F_{\alpha + 1})}$. Otherwise, let $C_{\xi + 1} = C_\xi$ and leave $x^\alpha_\xi$ undefined. If $\xi$ is
  limit ordinal, let $C_\xi = \bigcup_{\eta < \xi} C_\eta$. Note that all $C_\xi$ constructed here belongs to
  $R(\kappa, \lambda)$ due to \ref{baumkclosed} and the fact that $F_{\alpha + 1}$ is closed (i.e.\@ contains all its
  limit points).

  With $C_{\tau_\alpha}$ constructed, let $B_{\alpha + 1} \in R(\kappa, \lambda)$ be such that
  \begin{eqnarray*}
    |\mathrm{domain}\ B_{\alpha + 1} \setminus \mathrm{domain}\ B_\alpha| & = & \kappa, \\
    C_{\tau_\alpha} & \geq & B_{\alpha + 1}, \  \text{and} \\
    {B_\alpha \cap (E_\alpha \times 2 \times F_{\alpha + 1})} & = & {B_{\alpha + 1} \cap (E_\alpha \times 2 \times
    F_{\alpha + 1})}.
  \end{eqnarray*}
  The $B_{\alpha + 1}$ thus constructed satisfy \ref{bauminacunc}. Let $f_{\alpha + 1}$ be any function satisfying
  \ref{baumfalpha} and \ref{baumincreq}, and let $a_{\alpha + 1}
  > a_\alpha$ be such that \ref{baumaalpha} is satisfied (its existence is guaranteed by the validity of
  \ref{baumkclosed}). This is enough to conclude the recursive construction.

  Now, let $B_\kappa = \bigcup_{\alpha < \kappa} B_\alpha$ and $f = \bigcup_{\alpha < \kappa} f_\alpha$. We cannot
  say that $B_\kappa \in R(\kappa, \lambda)$ but, since $f$ maps $\kappa$ 1-1 onto $\mathrm{domain}\ B_\kappa$,
  \ref{baumextaalpha} implies that, for every $\alpha \in \mathrm{domain}\ B_\kappa$ and $\gamma$ such that
  $f^{-1}(\alpha) < \gamma < \kappa$, neither $(\alpha, 0, a_\gamma)$ nor $(\alpha, 1, a_\gamma)$ belongs to
  $B_\kappa$, so each element of $\mathrm{domain}\ B_\kappa$ has a closed unbounded subset of $\{a_\alpha: \alpha <
  \kappa\}$ able to satisfy \ref{baumrklreq3}, thus there exists $B'_\kappa \in R(\kappa, \lambda)$ such that
  $B'_\kappa \supseteq B_\kappa$.

  Let $X \in \mathfrak{M}$ be the set of all the $x^\alpha_\xi$ defined above. Since $|X| \leq \kappa$ in
  $\mathfrak{M}$, then $(X, {\prec})$ has order-type $< \kappa^+$ in both $\mathfrak{M}$, $\mathfrak{M}[G]$ and, once
  $(Y, {\prec})$ has order-type $\kappa^+$ in $\mathfrak{M}[G]$, there must be at least one element of $Z \setminus
  X$ belonging to $Y$. Let $x \in {Z \setminus X}$ and $B' \leq B'_\kappa$ be such that $B' \Vdash \check{x} \in
  \dot{Y}$ in $\mathfrak{M}$. Each $\alpha < \kappa$ has a $\xi < \tau_\alpha$ such that $B' \cap (E_\alpha \times 2
  \times F_{\alpha + 1}) = D_\xi$, then \ref{bauminacunc} implies that $x^\alpha_\xi$ is defined and $B' \leq
  B_{\alpha + 1} \cup D_\xi$. Consequently, $\mathrm{obj}\ B'$ contains a $\kappa$-sized subset of $X$, concluding
  this case.

  \textit{Case 2}: $\kappa$ uncountable and accessible. Here, we will also construct additionally the sequence
  $(a_\alpha: \alpha < \kappa)$ satisfying \ref{baumfandaalpha} to \ref{baumlimaalpha}, implying the validity of
  \ref{baumextaalpha} too. The constructions for ordinal limits and for $\alpha = 0$ are identical to the case above.

  In the $(\alpha + 1)$th step, however, since $|\wp(E_\alpha \times 2 \times F_{\alpha + 1})| = \kappa$ for most of
  $\alpha < \kappa$, we must rely on the validity of $\diamond_\kappa$ in $\mathfrak{M}$, which is equivalent to
  \begin{enumerate}[resume*=Baumgartner]
    \item \label{baumdiamondvar} There is a sequence $(S_\gamma: \gamma < \kappa)$ such that $S_\gamma \subseteq
        \gamma \times 2 \times \gamma$ for all $\gamma < \kappa$ and, for all $W \subseteq \kappa \times 2 \times
        \kappa$, $\{\gamma < \kappa: {W \cap {(\gamma \times 2 \times \gamma)}} = S_\gamma\}$ is stationary in
        $\kappa$ \cite[(22), p.\@ 430]{baumgartner_almost-disjoint_1976}.
  \end{enumerate}
  Let $(S_\gamma: \gamma < \kappa) \in \mathfrak{M}$ be like above. Here, instead of \ref{bauminacunc}, we will
  construct additionally $x_\alpha$ for (not necessarily all) $\alpha < \kappa$ satisfying:
  \begin{enumerate}[resume*=Baumgartner]
    \item \label{baumacunc} If there exist $B' \leq B_{\alpha + 1}$ and $x \in Z$ different from all defined
        $x_\xi$ with $\xi < \alpha$, both satisfying
        \begin{align*}
          (f_\alpha(\gamma), i, a_\alpha) \not\in B' \  \text{for all} \  \gamma < \alpha, \  i < 2; \\
          \{(\xi, i, \eta) \in {\alpha \times 2 \times \alpha}: (f_\alpha(\xi), i, a_\eta) \in {B'}\} = S_\alpha & \
          \text{and} \  B' \Vdash \check{x} \in \dot{Y};
        \end{align*}
        then $x_\alpha$ is defined and
        $${B_{\alpha + 1} \cup \{(f_\alpha(\xi), i, a_\eta): (\xi, i, \eta) \in S_\alpha\}} \Vdash
        {(x_\alpha)\,{\check{}} \in \dot{Y}}.$$
  \end{enumerate}
  The construction here is simple: If there exists $B' \leq B_\alpha$ with $|\mathrm{domain}\ B_{\alpha + 1}
  \setminus \mathrm{domain}\ B_\alpha| = \kappa$ satisfying the requirements above, let
  $$B_{\alpha + 1} = {(B' \setminus {(E_\alpha \times 2 \times F_{\alpha + 1})})} \, \cup \, {(B_\alpha \cap
  {(E_\alpha \times 2 \times F_{\alpha + 1})})}$$ and let $x_\alpha = x$. Otherwise, let $B_{\alpha + 1} \leq
  B_\alpha$ be arbitrary such that
  \begin{eqnarray*}
    |\mathrm{domain}\ B_{\alpha + 1} \setminus \mathrm{domain}\ B_\alpha| & = & \kappa; \  \text{and}\\
    {B_\alpha \cap (E_\alpha \times 2 \times F_{\alpha + 1})} & = & {B_{\alpha + 1} \cap (E_\alpha \times 2 \times
    F_{\alpha + 1})},
  \end{eqnarray*}
  and leave $x_\alpha$ undefined. The $f_{\alpha + 1}$ and $a_{\alpha + 1}$ are constructed as in Case 1.

  With the recursive construction over $\kappa$ done, let $B_\kappa = \bigcup_{\alpha < \kappa} B_\alpha$ and $f =
  \bigcup_{\alpha < \kappa} f_\alpha$. Arguing like in the previous case, there exists $B'_\kappa \supseteq B_\kappa$
  belonging to $R(\kappa, \lambda)$ even if $B_\kappa$ does not belong to.

  Let $X$ denote the set of all $x_\alpha$ defined in the recursive construction. Then, like in the previous case,
  $X$ belongs to $\mathfrak{M}$ and has cardinality $\leq \kappa$ in it, implying that $(X, {\prec})$ has order-type
  $< \kappa^+$ in both $\mathfrak{M}$ and $\mathfrak{M}[G]$. Thus, in $\mathfrak{M}[G]$, $Y$ must have an element of
  $Z \setminus X$.

  Let $x \in {Z \setminus X}$ and $B' \leq B'_\kappa$ be such that $B' \Vdash \check{x} \in \dot{Y}$ in
  $\mathfrak{M}$. We will work strictly in $\mathfrak{M}$ until the end of this case, so we can assume the validity
  of \ref{baumdiamondvar} for the sequence $(S_\gamma: \gamma < \kappa)$ used in the recursion.

  For every $\alpha < \kappa$, define the set
  $$U_\alpha = \{\gamma < \kappa: {(f(\alpha), 0, \gamma), (f(\alpha), 1, \gamma)} \not\in B'\} \quad
  \text{\cite[p.\@ 431]{baumgartner_almost-disjoint_1976}},$$ then each $U_\alpha$ is closed unbounded by definition.
  Define also the sets
  \begin{eqnarray*}
    U &=& \{\alpha < \kappa: a_\alpha \in U_\gamma \  \text{for all} \  \gamma < \alpha\}; \\
    V &=& \{(\alpha, i, \gamma) \in \kappa \times 2 \times \kappa: (f(\alpha), i, a_\gamma) \in B'\}; \\
    S &=& \{\alpha < \kappa: {V \cap {(\alpha \times 2 \times \alpha)}} = S_\alpha\} \quad \text{\cite[p.\@
    432]{baumgartner_almost-disjoint_1976}}
  \end{eqnarray*}
  Since $\{a_\alpha: \alpha < \kappa\}$ is closed unbounded, $U$ will also be closed unbounded and, due to
  \ref{baumdiamondvar}, $S$ is stationary, so $U \cap S$ is $\kappa$-sized (a stationary set indeed). Once $x$ does
  not belong to $X$, $x_\alpha$ is defined for each $\alpha \in U \cap S$ and also $B' \leq B_{\alpha + 1} \cup
  S_\alpha$, implying that $\mathrm{obj}\ B'$ contains a $\kappa$-sized subset of $X$.

  \textit{Case 3}: $\kappa = \aleph_0$. Once \ref{baumkclosed} is false here, the proof needs some changes. Since any
  $\alpha < \aleph_0$ is finite, the construction can avoid completely any limit ordinal step. However, we cannot use
  a sequence $(a_\alpha: \alpha < \kappa)$ satisfying \ref{baumfandaalpha} to \ref{baumextaalpha}.

  Instead of $(a_\alpha: \alpha < \kappa)$, we will construct additionally along the recursion an 1-1 function $g:
  [\kappa]^2 \rightarrow \kappa$ satisfying:
  \begin{enumerate}[resume*=Baumgartner]
    \item $F_\alpha = \{g(H): H \in [\alpha]^2\}$;
    \item If $\gamma < \alpha < \kappa$, then $(f_\alpha(\gamma), 0, g(\{\gamma, \alpha\})), (f_\alpha(\gamma), 1,
        g(\{\gamma, \alpha\})) \not\in B_\xi$ for all $\xi \leq \alpha$.
  \end{enumerate}
  Such a definition of $F_\alpha$ ensure that $|\wp(E_\alpha \times 2 \times F_{\alpha + 1})| < \aleph_0$ for all
  $\alpha < \aleph_0$, then the proof here will follow the same idea as in Case 1, except that there will be no limit
  ordinal step. Both statements above together with \ref{baumrecsucreq} imply
  \begin{enumerate}[resume*=Baumgartner]
    \item \label{baumexta0aalpha} If $\gamma < \alpha < \kappa$, then $(f_\alpha(\gamma), 0, g(\{\gamma,
        \alpha\})), (f_\alpha(\gamma), 1, g(\{\gamma, \alpha\})) \not\in B_\xi$ for all $\xi < \kappa$.
  \end{enumerate}
  In any $\alpha$th step of the recursion, instead of $a_\alpha$, we will define the set
  $$\{g(\{\xi, \alpha\}): \xi < \alpha\} = F_{\alpha + 1} \setminus F_\alpha,$$
  which will be empty iff $\alpha = 0$.

  Let $B_0 = B$ and define $f_0$ like in Cases 1 and 2, which are enough to conclude the step $\alpha = 0$. In the
  $(\alpha + 1)$th step, we will proceed as in Case 1, with $(D_\xi: \xi < \tau_\alpha)$ defined likewise,
  $x^\alpha_\xi$ constructed for (not necessarily all) $\alpha < \kappa$ and $\xi < \tau_\alpha$ satisfying
  \ref{bauminacunc} ($|\tau_\alpha| < \kappa$ here too), together with the construction of $(C_\xi: \xi \leq
  \tau_\alpha)$ in order to define $B_{\alpha + 1}$ through $C_{\tau_\alpha}$, and $f_{\alpha + 1}$ will be
  constructed identically to Case 1. The definition of $C_\xi$ for $\xi$ limit ordinal will not be used.

  The construction of $F_{\alpha + 1} \setminus F_\alpha$ will follow the same rule for every $\alpha \neq 0$. Once
  defined $B_\alpha$ and $f_\alpha$, the definition of $R(\kappa, \lambda)$ guarantees that, for any $\gamma <
  \alpha$, there is infinitely many $\xi < \kappa$ such that $(f_\alpha(\gamma), 0, \xi), (f_\alpha(\gamma), 1, \xi)
  \not\in B_\alpha$. Then, once defined $g(\{\eta, \alpha\})$ for all $\eta < \gamma < \alpha$, choose one of the
  $\xi$ above being different from any other already in $F_{\alpha + 1}$ (i.e.\@ different from any $g(\{\zeta,
  \sigma\})$ with either $\sigma < \alpha$ or $\sigma = \alpha$ and $\zeta < \gamma$) and let it be $g(\{\gamma,
  \alpha\})$. This concludes the recursive construction over $\kappa = \aleph_0$.

  Now, let $B_\kappa = \bigcup_{\alpha < \kappa} B_\alpha$, $f = \bigcup_{\alpha < \kappa} f_\alpha$ and $F_\kappa =
  \bigcup_{\alpha < \kappa} F_\alpha$. Like in the previous two cases, it is not necessarily true that $B_\kappa \in
  R(\kappa, \lambda)$ here too, but \ref{baumexta0aalpha} implies that, for every $\alpha \in \mathrm{domain}\
  B_\kappa$ and $\gamma$ such that $f^{-1}(\alpha) < \gamma < \kappa$, neither $(\alpha, 0, g(\{f^{-1}(\alpha),
  \gamma\}))$ nor $(\alpha, 1, g(\{f^{-1}(\alpha), \gamma\}))$ belongs to $B_\kappa$. Therefore, each element of
  $\mathrm{domain}\ B_\kappa$ has an infinite subset of $F_\kappa$ able to satisfy \ref{baumrklreq3}, consequently
  there exists $B'_\kappa$ such that $B'_\kappa \supseteq B_\kappa$.

  Let $X \in \mathfrak{M}$ be the set of all $x^\alpha_\xi$ defined above. $|X| \leq \kappa$ in $\mathfrak{M}$
  implies that $\text{order-type}\ (X, {\prec}) < \kappa^+$ in both $\mathfrak{M}$ and $\mathfrak{M}[G]$. Since
  $\text{order-type}\ (Y, {\prec}) = \kappa^+$ in $\mathfrak{M}[G]$, there must be $x \in {Z \setminus X}$ and $B'
  \leq B'_\kappa$ such that $B' \Vdash \check{x} \in \dot{Y}$ in $\mathfrak{M}$. Therefore, like in Case 1, we can
  prove that \ref{bauminacunc} implies $\mathrm{obj}\ B'$ contains a $\kappa$-sized subset of $X$.
\end{proof}

The only difference between the proof above and the corrected proof of Lemma 6.10 is that, after making $Z = \delta$,
the existence of at least one element of $\delta \setminus X$ in $Y$ is guaranteed by the fact that $\delta$ has
cofinality $\geq \kappa^+$ in $\mathfrak{M}$, implying that $\sup\ X < \delta$ in both $\mathfrak{M}$ and
$\mathfrak{M}[G]$. Aside from typographical errors and blurred characters, the proof issues presented in
\cite{baumgartner_almost-disjoint_1976} can be summarized as follows:
\begin{itemize}
  \item It did not define $F_\alpha$ and worked along the entire proof as if $F_\alpha = \alpha$, which turns it
      unable to guarantee that both \ref{baumrecsucreq} and \ref{baumaalpha} are able to imply \ref{baumextaalpha}.
      Compared with the proof presented here, it also used $\alpha$ in the place of $a_\alpha$ many times when
      $\kappa$ was assumed accessible, but this specific problem could be circumvented by restricting the argument
      to elements of the closed unbounded set $\{\alpha < \kappa: a_\alpha = \alpha\} \subseteq \{a_\alpha: \alpha
      < \kappa\}$.
  \item It mentioned that the Case 1, where $\kappa$ is assumed inaccessible, can include the case $\kappa =
      \aleph_0$, but it did not mention that the falsity of \ref{baumkclosed} for $\kappa = \aleph_0$ imposes
      changes in the argument. Here, the changes were made by replacing the sequence $(a_\alpha: \alpha < \kappa)$
      to the 1-1 function $g: [\kappa]^2 \rightarrow \kappa$.
  \item It also constructed the sequence $(C_\xi: \xi < \tau_\alpha)$ instead of $(C_\xi: \xi \leq \tau_\alpha)$.
      This led to some incongruences in the proof: the $x^\alpha_0$ is explicitly mentioned as always undefined
      and, if $\tau_\alpha = \xi + 1$ (being valid iff $\tau_\alpha$ is finite), then $x^\alpha_\xi$ is never
      defined too, inhibiting the validity of \ref{bauminacunc}.
\end{itemize}

With Lemma 6.10 proved, the first paragraph of the proof of Lemma 6.11 furnish us a result ensuring that cofinalities
are preserved in $\mathfrak{M}[G]$. Therefore, since $Y \in \mathfrak{M}[G]$ is unbounded in $Z \in \mathfrak{M}$ and
has order-type $\kappa^+$, which is regular in both $\mathfrak{M}$ and $\mathfrak{M}[G]$, then $\text{order-type}\
(Z, {\prec})$ has cofinality $\kappa^+$ in $\mathfrak{M}$. We will use this to conclude the proof of the proposition.
\begin{proof}[Proof for $\beta > \kappa$]
  Here, we will perform a slightly modified version of the recursive construction over $\kappa$ made in the case
  $\beta = \kappa$. Fix an enumeration $(\delta_\alpha: \alpha < \kappa)$ of $\beta$. The required changes can be
  summarized as follows:
  \begin{description}
    \item[$\kappa = \aleph_0$ or $\kappa$ is inaccessible] Instead of defining $x^\alpha_\xi \in Z$ that satisfies
        \ref{bauminacunc}, define $X^\alpha_\xi \subseteq Z$ for (not necessarily all) $\alpha < \kappa$ and $\xi <
        \tau_\alpha$ such that the following is satisfied:
        \begin{itemize}
          \item If there exists $B' \leq B_{\alpha + 1}$ such that $(\mathrm{obj}\ B', {\prec})$ has order-type
              $\delta_\alpha$ and $B' \cap (E_\alpha \times 2 \times F_{\alpha + 1}) = D_\xi$, then
              $X^\alpha_\xi$ is defined, $(X^\alpha_\xi, {\prec})$ has order-type $\delta_\alpha$ and
              $X^\alpha_\xi \subseteq \mathrm{obj}\ {B_{\alpha + 1} \cup D_\xi}$ (thus $B' \leq {B_{\alpha + 1}
              \cup D_\xi}$ is valid).
        \end{itemize}
    \item[$\kappa$ is accessible] Instead of defining $x_\alpha \in Z$ satisfying \ref{baumacunc}, we must define
        $X^\alpha_\xi$ for (not necessarily all) $\xi \leq \alpha$ satisfying:
        \begin{itemize}
          \item If there exists $B' \leq B_{\alpha + 1}$ such that
              \begin{align*}
                (f_\alpha(\gamma), i, a_\alpha) \not\in B' \  \text{for all} \  \gamma < \alpha, \  i < 2;\\
                \{(\xi, i, \eta) \in {\alpha \times 2 \times \alpha}: (f_\alpha(\xi), i, a_\eta) \in {B'}\} =
                S_\alpha; & \ \text{and} \\
                \text{order-type}\ (\mathrm{obj}\ B', {\prec}) = \delta_\xi;
              \end{align*}
              then $X^\alpha_\xi$ is defined, $(X^\alpha_\xi, {\prec})$ has order-type $\delta_\xi$, and also
              $$X^\alpha_\xi \subseteq \mathrm{obj}\ {(B_{\alpha + 1} \cup \{(f_\alpha(\xi), i, a_\eta): (\xi, i,
              \eta) \in S_\alpha\})}.$$ Note that the existence of such an $B'$ implies
              \begin{eqnarray*}
                S_\alpha & \supseteq & \{(\xi, i, \eta): (f_\alpha(\xi), i, a_\eta) \in B_\alpha \cap (E_\alpha \times
                2 \times F_{\alpha + 1})\} \  \text{and} \\
                B' & \leq & {B_{\alpha + 1} \cup \{(f_\alpha(\xi), i, a_\eta): (\xi, i, \eta) \in S_\alpha\}}.
              \end{eqnarray*}
        \end{itemize}
  \end{description}
  In order to perform this construction, we must do it in the following way depending on the respective case:
  \begin{description}
    \item[$\kappa = \aleph_0$ or $\kappa$ inaccessible] Construct additionally the sequence $(C_\xi: \xi \leq
        \tau_\alpha)$ of elements of $R(\kappa, \lambda)$ such that, for $\xi = 0$ or $\xi$ limit ordinal, $C_\xi$
        will be defined in the same way as in the case $\beta = \kappa$, and each $X^\alpha_\xi$ will be defined
        together with $C_{\xi + 1}$ as follows. If there exists $\bar{B} \in R(\kappa, \lambda)$ satisfying
        $$\bar{B} \leq C_\xi, \  \bar{B} \cap (E_\alpha \times 2 \times F_{\alpha + 1}) = D_\xi \  \text{and} \
        \text{order-type}\ (\mathrm{obj}\ \bar{B}, {\prec}) = \delta_\alpha,$$ then let $\mathrm{obj}\ \bar{B}$ be
        $X^\alpha_\xi$ and $C_{\xi + 1}$ be $\bar{B} \setminus {(E_\alpha
        \times 2 \times F_{\alpha + 1})}$. Otherwise, let $C_{\xi + 1} = C_\xi$ and leave $X^\alpha_\xi$ undefined.
    \item[$\kappa$ accessible] We shall proceed in the same way as above, but now with $\tau_\alpha = \alpha + 1$,
        $$D_\xi = \{(f_\alpha(\eta), i, a_\zeta): (\eta, i, \zeta) \in S_\alpha\}$$
        for every $\xi < \alpha + 1$ and the $(\mathrm{obj}\ \bar{B}, {\prec})$ here must have order-type
        $\delta_\xi$ in order to define $X^\alpha_\xi$ together with $C_{\xi + 1}$.
  \end{description}

  Once concluded the recursive construction and defined $B_\kappa$, we can conclude similarly to the case $\beta =
  \kappa$ the existence of a $B'_\kappa \in R(\kappa, \lambda)$ such that $B'_\kappa \supseteq B_\kappa$. Through the
  same methods presented at the end of Cases 1, 2 and 3 in the proof for $\beta = \kappa$, the conclusion we will get
  here is that, if
  $$\text{order-type}\ (\mathrm{obj}\ B', {\prec}) = \delta_\gamma < \beta \  \text{for some} \ B' \leq B'_\kappa,$$
  then there will be at least one defined $X^\xi_\eta$ such that $(X^\xi_\eta, {\prec})$ has order-type
  $\delta_\gamma$ and $X^\xi_\eta \subseteq \mathrm{obj}\ B'$, thus $X^\xi_\eta$ is a cofinal subset of
  $\mathrm{obj}\ B'$ with respect to ${\prec}$.

  Let $X \in \mathfrak{M}$ be the union of all $X^\xi_\eta$ defined in the recursive construction above. Since, in
  both $\mathfrak{M}$ and $\mathfrak{M}[G]$, the order-type of $(Z, {\prec})$ has cofinality $\kappa^+$ and $|X| \leq
  \kappa$ is valid, then $X$ is bounded in $Z$ with respect to ${\prec}$. Therefore, our assumptions at the beginning
  of the proposition's proof, i.e.\@ $Y \in \mathfrak{M}[G]$ being unbounded in $Z$ with respect to ${\prec}$, allow
  us to conclude the existence of an $y \in Z$ and a $B' \leq B'_\kappa$ such that $x \prec y$ for every $x \in X$
  and $B' \Vdash \check{y} \in \dot{Y}$. Since no $X^\xi_\eta$ can be cofinal in $\mathrm{obj}\ B'$ with respect to
  ${\prec}$, then $(\mathrm{obj}\ B', {\prec})$ must have order-type $\geq \beta$.

  The proof above depends on the fact that cofinalities of $\mathfrak{M}$ are preserved in $\mathfrak{M}[G]$ and it
  works for every ordinal $\beta \in \mathfrak{M}$ such that $|\beta| = \kappa$ in $\mathfrak{M}$, which is enough to
  conclude the proof for $\beta > \kappa$.
\end{proof}

\section*{Observations}

Compared with the proof presented here, Baumgartner \cite{baumgartner_almost-disjoint_1976} replaces $F_\alpha$ by
$\alpha$ and $a_\gamma$ by $\gamma$ consistently along the entire proof provided by it, which is why I consider the
issues presented there are not merely typographical errors.

I tried to look for errata available in the journal where \cite{baumgartner_almost-disjoint_1976} was published and
also in all articles I could find citing \cite{baumgartner_almost-disjoint_1976}, but I was not able to find neither
an explicit mention of issues in its proof, nor descriptions of corrections it needs.


\begin{thebibliography}{1}

\bibitem{baumgartner_almost-disjoint_1976} James~E. Baumgartner.
\newblock Almost-disjoint sets, the dense set problem and the partition
  calculus.
\newblock {\em Annals of Mathematical Logic}, 9(4):401--439, May 1976.

\end{thebibliography}
\end{document}